\documentclass[11pt]{amsart}
\usepackage[latin1]{inputenc}
\usepackage[active]{srcltx}
\usepackage{fullpage}
\usepackage{amsfonts,enumerate}
\usepackage[mathscr]{euscript}
\usepackage{epsfig}
\usepackage{latexsym}
\usepackage{leftindex}
\usepackage[all]{xy}
\usepackage{amssymb}
\usepackage{amsthm}
\usepackage{float}
\usepackage{amsmath}
\usepackage{amsxtra}
\usepackage{xcolor}
\usepackage[colorlinks]{hyperref}
\usepackage{fancyvrb}
\usepackage{tikz}
\usepackage[textwidth=2.5cm]{todonotes}
\hypersetup{citecolor=blue}

.tfm
.tfm

\theoremstyle{plain}
\newtheorem{proposition}{Proposition}[section]
\newtheorem{theorem}[proposition]{Theorem}
\newtheorem{corollary}[proposition]{Corollary}

\newtheorem{lemma}[proposition]{Lemma}
\newtheorem{definition}[proposition]{Definition}

\newtheorem*{thm*}{Theorem}
\newtheorem*{lemma*}{Lemma}

\newtheorem*{prop*}{Proposition}

\theoremstyle{definition}

\theoremstyle{remark}
\newtheorem{example}{Example}

\newtheorem{remark}{Remark}

\numberwithin{table}{section}

\DeclareMathOperator{\Frob}{Frob}

\DeclareMathOperator{\Gal}{Gal}

\DeclareMathOperator{\Jac}{J}

\newcommand{\GL}{{\rm GL}}

\newcommand{\HGM}{\mathcal H}
\newcommand{\hgm}{\HGM((a,b),(c,d)|z)}

\newcommand{\Mq}{{\mathcal M}_{(q,q,p)}^s}
\newcommand{\Mqs}{{\mathcal M}_{(q,q,p)}^{s_0}}
\newcommand{\Mp}{{\mathcal M}_{(p,p,q)}}
\newcommand{\M}{\mathcal M}

\def\ZZ{\mathbb Z}

\def\FF{\mathbb F}
\def\QQ{\mathbb Q}

\def\PP{\mathbb P}

\def\PP{\mathbb P}

\def\CC{\mathbb C}

\def\C{\mathcal C}

\def\<#1>{{\left\langle{#1}\right\rangle}}

\def\Z{{\mathbb Z}}             
\def\Q{{\mathbb Q}}             

\def\id#1{{\mathfrak{#1}}}      
\def\normid#1{{\norm{\id{#1}}}}
\DeclareMathOperator{\norm}{{\mathscr N}}

\DeclareMathOperator{\JacMot}{\mathbf J}

\usepackage{array}   
\newcolumntype{L}{>{$}l<{$}} 
\newcolumntype{C}{>{$}c<{$}} 

\def\gent{\chi_{\id{p}}}

\begin{document}
	
\title[Transformations of HGM]{On Transformation properties of hypergeometric motives  and Diophantine equations}
	
\author{Ariel Pacetti}

\address{Center for Research and Development in Mathematics and
  Applications (CIDMA), Department of Mathematics, University of
  Aveiro, 3810-193 Aveiro, Portugal} \email{apacetti@ua.pt}
\thanks{Funded by the Portuguese Foundation for
  Science and Technology (FCT) Individual Call to Scientific Employment Stimulus
  (https://doi.org/10.54499/2020.02775.CEECIND/CP1589/CT0032). This
  work was supported by CIDMA under the FCT
Multi-Annual Financing Program for R\&D Units}


\begin{abstract}
  Over the last two hundred years different transformation formulas for
  Gauss' hypergeometric function ${}_2F_1$ were discovered. The goal
  of the present article is to study their arithmetic analogue for the
  underlying hypergeometric motive. As an application, we show how these transformation properties can be used in the study of some Diophantine equations.
\end{abstract}
	
\maketitle

\section{Introduction}

Let $a,b,c$ be complex numbers, and let $z \in \CC$ with $|z|<1$. Recall the definition of Gauss' hypergeometric function
\begin{equation}
  \label{eq:Gauss}
  {}_2F_1(a,b;c|z):=\sum_{n=0}^\infty \frac{(a)_n(b)_n}{(c)_n n!} z^n,
\end{equation}
where for $n$ a non-negative integer, $(a)_n$ is the Pocchammer symbol
defined by
\[
  (a)_n = \begin{cases}
    1 & \text{ if }n=0,\\
    a(a+1)\cdots (a+n-1) & \text{ if }n>0.
    \end{cases}
\]

The hypergeometric function ${}_2F_1$ satisfies the differential equation
\begin{equation}
  \label{eq:diff-eq}
  z(1-z) \frac{d^2}{dz^2} + [c-(a+b+1)z]\frac{d}{dz} -ab=0.
\end{equation}

The differential equation has three singular points in $\PP^1$, namely
at the points $\{0,1,\infty\}$.  There are many transformation
formulas satisfied by Gauss' hypergeometric function (discovered mostly
during the XIX century by Gauss, Kummer, Goursat et al).

\begin{example}
  \label{ex:1}
  Formula 49 in \cite{MR1578088} (page 77), which matches formula
  $(38)$ in \cite{MR1508709} (page 119), states the following relation:
  \begin{equation}
  \label{eq:1}
 {}_2F_1\left(a,b;\frac{a+b+1}{2}|z\right) = {}_2F_1\left(\frac{a}{2},\frac{b}{2};\frac{a+b+1}{2}|4z(1-z)\right).
\end{equation}
\end{example}

\vspace{2pt}

With some manipulation, it is not hard to verify that both functions
satisfy the same differential equation.  Riemann realized that instead
of studying the differential equation itself, one should focus on the
so called \emph{monodromy representation}. Let
$z_0\in \CC \setminus \{0,1\}$ be a complex ``base point''. Let
$\{f_1,f_2\}$ be a basis of solutions of the differential equation. If
$\eta$ is a loop based at $z_0$, one can holomorphically extend the
solutions $f_1$ and $f_2$ along $\eta$ and obtain a new basis of
solutions. This gives a representation
\begin{equation}
  \label{eq:rep-monodromy}
  \rho : \pi_1(\PP^1\setminus\{0,1,\infty\},z_0) \to \GL_2(\CC),
\end{equation}
that sends a loop $\eta$ to the change of basis matrix.

With a little more generality, given $a,b,c,d$ complex numbers, consider the
differential equation
\begin{equation}
  \label{eq:diff-2}
  z^2(1-z)\frac{d^2}{dz^2}+z[c+d-1-(a+b+1)z]\frac{d}{dz}+(c-1)(d-1)-abz=0.
\end{equation}

This differential equation has a similar monodromy representation, that
we denote by $\rho$. The case $d=1$ corresponds to Gauss'
hypergeometric function, while the general case is a ``twist'' of it (as explained in \cite{GPV}).

For ease of notation, we denote by $\exp(a)$ the complex number
$e^{2\pi i a}$.

\begin{definition}
  We say that the parameters $(a,b),(c,d)$ are \emph{generic} if the
  sets $\{\exp(a),\exp(b)\}$ and
  $\{\exp(c),\exp(d)\}$ are disjoint.
\label{def:gen}
\end{definition}

In the present article we will only study parameters that are generic,
so from now we assume this is always the case. Under this assumption
the following properties hold:
\begin{enumerate}
\item The monodromy representation depends only on the values
  $a,b,c,d$ up to integral translation (as proven in \cite[Proposition
  2.5]{BH})
  
\item The monodromy representation is irreducible (as proven in
  \cite[Proposition 3.3]{BH}).
  
\item The monodromy representation is \emph{rigid} (as proven in
  \cite[Theorem 3.5]{BH}).
\end{enumerate}
Let us explain in more detail the notion of being rigid. Two
representations $\rho_1, \rho_2$ are \emph{isomorphic} if they are the
same up to a choice of a basis for the vector space. Equivalently,
once a basis for each vector space is chosen, they are isomorphic if
there exists a matrix $C \in \GL_2(\CC)$ such that
$\rho_1(\eta) = C\rho_2(\eta)C^{-1}$ for any loop $\eta$.

A direct application of Van Kampen's theorem proves that the
fundamental group $\pi_1(\PP^1\setminus\{0,1,\infty\},z_0)$ is
generated by three loops $\{\eta_0,\eta_1,\eta_\infty\}$, where
$\eta_i$ is a simple loop in counterclockwise direction around the
point $i$, for $i=0,1,\infty$, satisfying the condition
$\eta_0 \eta_1\eta_\infty$ being homotopically trivial. The
monodromy representation is then uniquely determined by the values
\[
M_0:=\rho(\eta_0),\qquad M_1:=\rho(\eta_1),\qquad M_\infty:=\rho(\eta_\infty).
  \]
  Two representations $\rho,\rho'$ are isomorphic if the triple of
  matrices $(M_0,M_1,M_\infty)$ are conjugated to
  $(M_0',M_1',M_\infty')$ (by an invertible matrix). The rigidity
  property means that it is enough for each $M_i$ being conjugated to
  $M_i'$ for $i=0,1,\infty$ for the two representations to be
  isomorphic (so the conjugacy class of each monodromy matrix determines the
  representation).
  
A result of Level (Theorem 1.1 of his Ph.D. thesis, or \cite[Theorem
3.5]{BH}) states that given generic parameters $(a,b),(c,d)$, there exists a (unique up to isomorphism) monodromy representation
whose representatives for the conjugacy monodromy matrices $M_i$ around the
point $i$, for $i=0,1,\infty$ are
\begin{equation}
  \label{eq:monodromy}
  M_0 =
  \begin{cases}
    \left(\begin{smallmatrix}
            \exp(-c) & 0\\
            0 & \exp(-d)\end{smallmatrix}\right) & \text{ if } c-d \not \in \ZZ,\\
    \left(\begin{smallmatrix}
            \exp(-c) & 1\\
            0 & \exp(-c)\end{smallmatrix}\right) & \text{ otherwise}.
  \end{cases}
  \quad
  M_1 =
  \begin{cases}
    \left(\begin{smallmatrix}
            1 & 0\\
            0 &\exp(c+d-a-b)\end{smallmatrix}\right) & \text{ if } a+b-c-d \not \in \ZZ,\\
    \left(\begin{smallmatrix}
            1 & 1\\
            0 & 1\end{smallmatrix}\right) & \text{ otherwise}.
  \end{cases}
\end{equation}
and
\[
  M_\infty =
  \begin{cases}
    \left(\begin{smallmatrix}
            \exp(a) & 0\\
            0 & \exp(b)\end{smallmatrix}\right) & \text{ if } a-b \not \in \ZZ,\\
    \left(\begin{smallmatrix}
            \exp(a) & 1\\
            0 & \exp(a)\end{smallmatrix}\right) & \text{ otherwise}.
  \end{cases}
\]
The monodromy representation attached to (the differential equation of) an
hypergeometric series is unramified outside $\{0,1,\infty\}$ and the
monodromy matrix at $1$ is a pseudo-reflexion (a matrix whose eigenvalue $1$
has an eigenspace of codimension $1$). Any monodromy representation
satisfying these two property comes from a differential equation as
(\ref{eq:diff-2}). We will use the term \emph{hypergeometric}
monodromy representation for them.

\begin{definition}
Let $\varphi : \PP^1 \to \PP^1$ be a continuous map, and let $\rho$ be
an hypergeometric monodromy representation. The \emph{pullback} of
$\rho$ by $\varphi$ is the monodromy representation
 \begin{equation}
   \label{eq:pullback}
   \varphi^*(\rho): \pi_1\left(\PP^1 \setminus \varphi^{-1}(\{0,1,\infty\})\right) \to \GL_2(\CC),
 \end{equation}
 defined by $\varphi^*(\rho)(\eta) = \rho
 (\varphi(\eta))$.
\end{definition}  

A priory the pullback of an hypergeometric representation needs not be hypergeometric (as for example it might ramify in more than three points).

 \begin{definition}
   An \emph{hypergeometric relation} is a pair $(\varphi,\rho)$
   consisting of an algebraic map $\varphi: \PP^1 \to \PP^1$ and an
   hypergeometric monodromy representation $\rho$ satisfying that
   $\varphi^*(\rho)$ is also hypergeometric.
 \end{definition}

 It seems like a natural question to study the following problem.
 
 \vspace{4pt}
 \noindent {\bf Problem:} determine/classify hypergeometric relations
 $(\varphi,\rho)$.
 \vspace{4pt}

 This problem (presented in a little different way) was intensively
 studied by Kummer (in \cite{MR1578088}) and Goursat (in
 \cite{MR1508709}). Goursat obtained a complete list of algebraic
 transformations satisfied by Gauss' hypergeometric function for small
 degrees. For a modern approach (including historical references) see
 the article \cite{MR2547100}.

 \setcounter{example}{0}
 \begin{example}[continued]
   Let us study Example~\ref{ex:1} in more detail.  Let
   $\varphi : \PP^1 \to \PP^1$ be the degree $2$ cover sending
   $z \to 4z(1-z)$. To compute the ramification points, we just
   compute the discriminant of the polynomial $4z(1-z)-t$, which
   equals $16(t-1)$. Then the cover is ramified precisely at the
   points $1$ and $\infty$. 

   The preimage of the ramified points are $1/2$ and $\infty$
   respectively. The preimage of $0$ consists of the two points
   $\{0,1\}$. Let $\rho$ denote the hypergeometric monodromy
   representation with parameters
   $\left(\frac{a}{2},\frac{b}{2}\right)\left(\frac{a+b+1}{2},1\right)$. Then
   the pullback $\varphi^*(\rho)$ is unramified outside the set
   $\{0,1,1/2,\infty\}$.

   A representative for the conjugacy class of $\rho$ at $\eta_1$
   is 
 \[
M_1 =
\begin{pmatrix}
  -1 & 0 \\
  0 & 1
\end{pmatrix}.
\]
Since $1/2$ is a ramified point,
$\varphi^*(\rho)(\eta_{1/2}) = M_1^2 =\left(\begin{smallmatrix}1&0\\0&
                                                                       1\end{smallmatrix}\right)$, so the pullback is unramified at
$1/2$. On the other hand,
\[
  \varphi^*(\rho)(\eta_0)=  \varphi^*(\rho)(\eta_1)=\rho(0)=
  \begin{pmatrix}
    \exp(\frac{a+b+1}{2}) & 0\\
    0 & 1
  \end{pmatrix}
\]
The fact that $\varphi^*(\rho)$ at $1$ is a pseudo-reflection implies
that the pullback is an hypergeometric monodromy representation, so
$(\varphi,\rho)$ is an hypergeometric relation. We can also compute
the missing monodromy matrix
\[\varphi^*(\rho)(\eta_\infty) = \rho(\eta_\infty)^2=
  \begin{pmatrix}
    \exp(a) & 0\\
    0 & \exp(b)
  \end{pmatrix}.
\]
The rigidity property then implies that the monodromy representation
of $\varphi^*(\rho)$ matches the monodromy representation with
parameters $(a,b),(\frac{a+b+1}{2},1)$ as expected.   
 \end{example}
 Knowing that the two monodromy representations are the same, does not
 necessarily imply a simple relation between the corresponding Gauss'
 hypergeometric functions, since the space of solutions is two
 dimensional! It does however imply that one hypergeometric series is
 a linear combination of other two ones, but an equality
 like~(\ref{eq:1}) is not always true, and when it does hold, the proof requires
 some extra work.
 \begin{example}
\label{ex:2}
   Consider the transformation formula obtained from
   \cite{MR1578088} (formula 67 in page 81):

\begin{multline}
  \label{eq:67}
  {}_2F_1\left(a,b;\frac{a+b+1}{2}|z\right)=\frac{\cos(a-b)\frac{\pi}{2}}{\cos(a+b)\frac{\pi}{2}} \; {}_2F_1\left(a,b;\frac{a+b+1}{2}|1-z\right) +\\
  \frac{\Gamma(\frac{a+b-1}{2}) \Gamma(\frac{a+b-3}{2})}{\Gamma(a-1)\Gamma(b-1)} (1-z)^{\frac{1-a-b}{2}}{}_2F_1\left(\frac{a-b+1}{2},\frac{b-a+1}{2};\frac{3-a-b}{2}|1-z\right).
\end{multline}
The monodromy matrices of the hypergeometric series ${}_2F_1(a,b,\frac{a+b+1}{2})$ at $0, 1$ and $\infty$ are
\[
M_0=
\begin{pmatrix}
  \exp(-\frac{a+b+1}{2}) & 0\\
  0 & 1
\end{pmatrix},
\qquad
M_1=
\begin{pmatrix}
  \exp(-\frac{a+b+1}{2}) & 0\\
  0 & 1
\end{pmatrix},
M_{\infty}=
\begin{pmatrix}
  \exp(a) & 0\\
  0 & \exp(b)
\end{pmatrix}.
\]

Since the monodromy matrices at $0$ and at $1$ are the same, the map
that swaps $0$ and $1$ (sending $z$ to $1-z$) preserves the monodromy
representation. Then the first hypergeometric summand on the right
hand side satisfies the same differential equation as the left hand
side function. The hypergeometric series appearing in the second term
on the right has monodromy matrices
\[
  N_0=
  \begin{pmatrix}
    \exp(\frac{a+b+1}{2}) & 0\\
    0 & 1
  \end{pmatrix},
  \quad
  N_1=
  \begin{pmatrix}
    \exp(-\frac{a+b+1}{2}) & 0\\
    0 & 1
  \end{pmatrix},
  \quad
  N_{\infty}=
  \begin{pmatrix}
    \exp(\frac{a-b+1}{2}) & 0\\
    0 & \exp(\frac{b-a+1}{2})
  \end{pmatrix}.
\]
We need to compose with the map that swaps the monodromy at $0$ and
$1$ (so we get the right monodromy matrix at $0$) and then twist by
the character ramified at $\{1,\infty\}$ (corresponding to the term
$(1-z)^{\frac{1-a-b}{2}}$). This multiplies the monodromy matrix at
$1$ (namely $N_0$ because of the swap) by $\exp(-\frac{a+b+1}{2})$ and
the monodromy matrix at $\infty$ by $\exp(\frac{1+a+b}{2})$. Then the
monodromy representation also coincides with the previous one.
However, in this case, a linear combination of the two solutions is
needed to get a match of hypergeometric series.
\end{example}

\vspace{3pt}

In the present article we are interested on the \emph{arithmetic} side
of the previous transformation formulas. Suppose that the parameters
$(a,b),(c,d)$ besides being generic are also rational numbers. Let $N$
be their least common denominator. Then (as proved in \cite{GPV})
there exists an hypergeometric motive $\hgm$ defined over
$F:=\Q(\zeta_N)$ having the hypergeometric series as a period.

The profinite completion of the fundamental group
$\pi_1\left(\PP^1 \setminus \varphi^{-1}(\{0,1,\infty\})\right)$ is
isomorphic to $\Gal(\Omega/\overline{\Q}(z))$, where $\Omega$ denotes
the maximal extension of $\overline{\Q}(z)$ unramified outside
$\{0,1,\infty\}$ (see \cite[Theorem 6.3.1]{MR2363329}). Then for each
prime ideal $\id{p}$ of $F$ there is a continuous Galois
representation
\[
  \rho_{\id{p}}: \Gal(\overline{\Q(z)}/\overline{\Q}(z)) \to
\GL_2(F_{\id{p}}),
\]
``extending'' the monodromy representation (after identifying the
group $\pi_1(\PP^1\setminus \{0,1,\infty\},z_0)$ with a subgroup of
$\Gal(\Omega/\overline{\Q}(z))$). By \cite[Remark 6.28]{GPV} this
``geometric'' representation extends to an ``arithmetic'' one, namely
for each prime ideal $\id{p}$ of $F$ there exists a Galois representation
\begin{equation}
  \label{eq:gal-rep}
  \rho_{\id{p}}: \Gal(\overline{\Q(z)}/F(z)) \to
\GL_2(F_{\id{p}}),
\end{equation}
extending the monodromy representation. The goal of the present
article is to relate the classical hypergeometric series relations
with relations between the associated hypergeometric motives.

\section{Hypergeometric motives}
\label{sec:HGM}

The hypergeometric motive $\hgm$ defined in \cite{GPV} appears (up to
a twist by a Hecke character) in the middle cohomology of Euler's
curve. Consider the following problem:

\vspace{3pt}
\noindent {\bf Problem:} Let $(a,b),(c,d)$ and
$(\alpha,\beta),(\gamma,\delta)$ be two pairs of generic rational
parameters. How to determine if the hypergeometric motives $\hgm$ and
$\HGM((\alpha,\beta),(\gamma,\delta)|z)$ are isomorphic?
\vspace{3pt}

Let us start studying a related problem.
\begin{proposition}
  Let $H$ be a normal subgroup of $G$ and let $\rho:H \to \GL_n(L)$ be
  an irreducible representation of $H$. If the representation extends
  to an $n$-dimensional representation of $G$, then the extension is
  unique up to a twist by a character of $G/H$.
  \label{prop:unicity}
\end{proposition}

\begin{proof}
  The results follows from an easy application of Schur's
  lemma. Suppose that $\rho_i:G \to \GL_n(K)$, $i=1,2$ are two
  extensions of $\rho$. Then for any $g \in G$ and any $h \in H$, the
  equality
  \[
    \rho_1(g)\rho_1(h)\rho_1(g)^{-1}=\rho_1(ghg^{-1})=\rho_2(ghg^{-1})=\rho_2(g)\rho_2(h)\rho_2(g)^{-1},
  \]
  implies that $\rho_2(g^{-1})\rho_1(g)$ must commute with $\rho$,
  hence by Schur's lemma it is a scalar matrix of the form
  $\chi(g) \cdot 1_n$, for some $\chi(g) \in L^\times$ (where $1_n$
  denotes the $n \times n$ identity matrix). The fact that $\rho_1$
  and $\rho_2$ are representations implies that $\chi$ is a character,
  and the fact that they are extensions of $\rho$ implies that $\chi$ is trivial
  on $H$.
\end{proof}

\begin{corollary}
\label{coro:unicity}
  Let $K$ be a number field and let
  $\rho_i:\Gal(\overline{\Q(t)}/K(t)) \to \GL_2(\overline{\Q_p})$ for
  $i=1,2$ be two irreducible continuous representations. Suppose that the following two conditions hold:
  \begin{enumerate}
  \item   $\rho_1|_{\Gal_{\overline{\Q}(t)}} \simeq
    \rho_2|_{\Gal_{\overline{\Q}(t)}}$,
    
  \item The restriction $\rho_1|_{\Gal_{\overline{\Q}(t)}}$ is irreducible.
  \end{enumerate}
  Then there exists
  $\chi:\Gal(\overline{\Q}/K) \to \overline{\Q_p}^\times$ such that
  $\rho_1 \simeq \rho_2 \otimes \chi$.
\end{corollary}

\begin{proof}
  Just take $H:=\Gal(\overline{\Q(t)}/\overline{\Q}(t))$ and
  $G:=\Gal(\overline{\Q(t)}/K(t))$ in the previous proposition and
  recall that
  $\Gal(\overline{\Q}(t)/K(t)) \simeq \Gal(\overline{\Q}/K)$.
\end{proof}

The corollary gives an answer to our problem: the two hypergeometric
motives are isomorphic if and only if the following two properties
hold:
\begin{enumerate}
\item The monodromy representations of both parameters are isomorphic.
\item At one specialization of the parameter the motives are isomorphic.  
\end{enumerate}

One implication is clear: if the motives are isomorphic, (2) must
hold, and also the restriction of their Galois representations to
$\pi_1\left(\PP^1 \setminus \varphi^{-1}(\{0,1,\infty\})\right)$ must
also be isomorphic, so (1) holds. For the other implication, the first
property assures that both motives have the same geometric
representation, so by Corollary~\ref{coro:unicity} they differ by a
twist which does not depend on the parameter. Then the twist is
determined by any specialization.

Verifying whether condition (1) holds is a linear algebra
problem in most cases. To check whether condition (2) holds is
more challenging.

\subsection{Finite hypergeometric sums}
Let $N>1$ be an integer and $F=\Q(\zeta_N)$. Let $\id{p}$ be a prime
ideal in $F$ prime to $N$ and let $\FF_q$ denote its residue
field. Let $\psi$ be an additive character of $\FF_q^\times$. For $\omega$ a character of $\FF_q^\times$, denote by $g(\psi,\omega)$ Gauss' sum
\begin{equation}
  \label{eq:gauss-sum}
  g(\psi,\omega):=\sum_{x \in \FF_q^\times} \omega(x)\psi(x).
\end{equation}
Following \cite{MR0051263}, for $x$ an integer prime to $\id{p}$, let
$\gent(x)$ denote the $N$-th root of unity congruent to
$x^{(q-1)/N}$ modulo $\id{p}$. Extend the definition by setting
$\gent(x)=0$ if $\id{p} \mid x$. This determines a character
\begin{equation}
  \label{eq:gen-defi}
\gent: (\Z[\zeta_N]/\id{p})^\times \to \CC^\times.  
\end{equation}
\begin{definition}
  Let $(a,b),(c,d)$ be generic rational parameters, and let $N$ be
  their least common denominator. Let $\id{p}$ be a prime ideal of
  $F$. For $z \in \FF_q$, define the \emph{finite hypergeometric sum}
  $H_{\id{p}}((a,b),(c,d)|z)$ by
  \begin{equation}
    \label{eq:finite-hyp}
    \frac{1}{1-q}\sum_\omega \frac{g(\psi,\gent^{-aN}\omega)g(\psi,\gent^{-bN}\omega)g(\psi,\gent^{cN}\omega^{-1})g(\psi,\gent^{dN}\omega^{-1})}{g(\psi,\gent^{-aN})g(\psi,\gent^{-bN})g(\psi,\gent^{cN})g(\psi,\gent^{dN})}\omega(z),
  \end{equation}
  where the sum runs over characters of $\FF_q^\times$.
\end{definition}
The definition of the finite hypergeometric sum does not depend on the
choice of the additive character.
Let $z \in \Q$, with $z \neq 0,1$. Let $\id{p}$ a prime ideal of $F$
not dividing $N$ such that $v_{\id{p}}(z(z-1))=0$. Then the motive
$\hgm$ is unramified at $\id{p}$ and the trace of the Frobenius
element $\Frob_{\id{p}}$ acting on $\hgm$ equals
$H_{\id{q}}((a,b),(c,d)|z)$ (as proven in \cite[Theorem 6.29]{GPV}).

Then we can give a more combinatorial criteria to determine whether
two hypergeometric motives are isomorphic.

\begin{enumerate}
\item Determine whether the monodromy representations of both
  parameters are isomorphic.
  
\item Prove that there exists $z \in \Q$ satisfying that for all
  primes $\id{p}$ of $F$ but finitely many (or outside a set of
  density zero) the following equality holds:
  \[
    H_{\id{p}}((a,b),(c,d)|z)=H_{\id{p}}((\alpha,\beta),(\gamma,\delta)|z).
    \]
\end{enumerate}

\subsection{Special values}
\label{sec:Special}
Although the specializations $z=0, 1, \infty$ give singular motives,
in some cases it still makes sense to compute the value of a Frobenius
element at it. This flexibility allows in many concrete cases to prove
condition (2). Recall the following definition (from
\cite{MR0051263} and \S 5 of \cite{GPV}).

\begin{definition}
  \label{def:jacobi-motive}
  Let ${\bf a}=(a_1,\ldots,a_r)$ and ${\bf b}=(b_1,\ldots,b_s)$ be two
  sets of rational numbers and let $N$ be their least common
  denominator. The Jacobi motive attached to ${\bf a},{\bf b}$ at a
  prime ideal $\id{p}$ of $F=\QQ(\zeta_N)$ is defined as
  \begin{equation}
    \label{eq:Jac-motive}
    \JacMot({\bf a, b})(\id{p}) = (-1)^{r+s+1} \frac{g(\psi,\gent^{Na_1})\cdots
    g(\psi,\gent^{Na_r})g(\psi,\gent^{N\sum_j b_j-N\sum_i a_i})}{g(\psi,\gent^{Nb_1})\cdots g(\psi,\gent^{Nb_s})}.
  \end{equation}
\end{definition}

\begin{lemma}
  \label{lemma:spec-0} Let $(a,b),(c,1)$ be rational generic
  parameters. Let $\varphi(z)$ be a rational function vanishing at
  $z=0$ with order $v > 0$. Then if
  $vc\not \in \Z$, for any prime ideal $\id{p}$ of $F$ not dividing
  $N$, the trace of $\Frob_{\id{p}}$ on $\HGM((a,b),(c,1)|\varphi(z))$
  specialized at $z=0$ equals $1$.
\end{lemma}

\begin{proof}
  The proof follows the lines of \cite[Appendix A]{GPV} (see also
  \cite{LP}). Write $\varphi(z)=z^v \tilde{\varphi}(z)$, so
  $\tilde{\varphi}(0)\neq 0$.  Our motive is part of the middle
  cohomology of Euler's curve
  \[
y^N = x^A(1-x)^B (1-z^v\tilde{\varphi}(z)x)^C,
\]
where $A=-bN$, $B=(b-c)N$ and $C=aN$. The semistable model has two
components, namely
\[
\C_1:y^N=x^{A}(1-x)^B,
\]
and
\[
  \C_2:z^{v(A+B)}y^N = (-1)^Bx^{A+B}(1-\tilde{\varphi}(0)x)^C.
\]
The second curve is a twist by $\sqrt[N]{z^{cvN}}$ of a non-singular
curve (both curves have positive genus and complex
multiplication). Since $cv$ is not an integer, the action of inertia
at $z=0$ on $\C_2$ is non-trivial hence the only contribution comes
from $\C_1$. Let $\id{p}$ be a prime ideal of $F$ with norm $q$. Since
$\C_1$ has complex multiplication, the $1$-dimensional
contribution at a prime ideal $\id{p}$ equals (by \cite[Appendix A]{GPV})
\begin{equation}
  \label{eq:ap-0}
  -\frac{g(\psi,\gent^{bN})g(\psi,\gent^{(c-b)N})}{g(\psi,\gent^{cN})} =\JacMot((b,c-b),(c))(\id{p}).  
\end{equation}

The motive $\hgm$ (defined in \cite[Definition 6.8]{GPV}) is the twist of the contribution coming from Euler's curve times the character
\begin{equation}
  \label{eq:normalization}
\JacMot((-a,-b,c,1),(c-b,-a))^{-1}(\id{p})\gent(-1)^{bN} = \JacMot((-b,c)(c-b))^{-1}(\id{p})\gent(-1)^{bN}.  
\end{equation}
  Recall that if $\chi$ is a non-trivial character in $\FF_q^\times$,
  then
\begin{equation}
  \label{eq:inversion}
g(\psi,\chi)=\chi(-1)\frac{q}{g(\psi,\chi^{-1})}.
\end{equation}
Using the genericity condition and that $c$ is not an integer
we get the following relation for (\ref{eq:ap-0})
\begin{multline*}
  \JacMot((b,c-b)(c))(\id{p})=-\frac{g(\psi,\gent^{bN})g(\psi,\gent^{(c-b)N})}{g(\psi,\gent^{cN})}=  
  -\frac{\gent(-1)^{bN}g(\psi,\gent^{bN})g(\psi,\gent^{-cN})}{g(\psi,\gent^{(b-c)N})}=\\
  \JacMot((-b,c),(c-b))(\id{p})\gent(-1)^{bN}.
\end{multline*}
Then the product of (\ref{eq:ap-0}) and (\ref{eq:normalization}) equals $1$ as claimed.
\end{proof}

\begin{lemma}
  Let $a,b$ be rational numbers which are not integers and let $N$ be
  their least common denominator. Let $\varphi(z)$ be a rational
  function vanishing at $0$. Then for any prime ideal $\id{p}$ of $F$
  not dividing $N$, the trace of $\Frob_{\id{p}}$ on
  $\HGM((a,b),(1,1)|\varphi(z))$ specialized to $z=0$ equals $1$.
  \label{lemma:spec0-1}
\end{lemma}

\begin{proof}
  The proof follows the idea of \cite{LP} (see \S 2 of loc. cit. for
  details).  Let $N$ be the least common multiple of the denominators
  of $a,b$ and let $F=\Q(\zeta_N)$. Set
  $\varphi(z)=z^v\tilde{\varphi}(z)$, with $\tilde{\varphi}(0)\neq
  0$. Euler's curve for this parameters has equation
  \[
    \C:y^N=x^{-bN}(\varphi(z)-x)^{bN}(1-x)^{aN}.
  \]
  Assume that $\gcd(N,bN)=\gcd(N,aN)=1$ (the other case is a little
  more technical, see \cite{LP}). The semistable model of $\C$ consists
  of the two (genus zero) irreducible curves
  \[
    \C_1:y^N=x^{-bN}(\tilde{\varphi}(0)-x)^{bN}.
  \]
  and
  \[
    \C_2:y^N=(-1)^{bN}(1-x)^{aN}
  \]
  The two curves intersect in $N$ points defined over the extension
  $F(\sqrt[N]{(-1)^{bN}})$ (independent of the function
  $\varphi$). The Jacobian of Euler's curve matches the first
  cohomology group of the just described component graph tensored with
  the $2$-dimensional Steinberg representation.  Then if $\id{p}$ is a
  prime ideal of $F$ not dividing $N$, the trace of Frobenius at the
  $\zeta_N$-eigenpart of the Jacobian matches the value of the
  character giving the extension $F(\sqrt[N]{(-1)^{bN}})/F$ at
  $\id{p}$, whose value is $\gent(-1)^{bN}$. Since $c=d=1$, the
  normalization (\ref{eq:normalization}) also takes the value takes
  the same value $\gent(-1)^{bN}$ hence the quotient is trivial.
\end{proof}

Recall the following transformation formula

\begin{proposition}
  The motives $\HGM((a,b),(c,d)|z)$ and $\HGM((-c,-d),(-a,-b)|z^{-1})$
  are isomorphic.
  \label{prop:switch}
\end{proposition}

\begin{proof}
  See \cite[Proposition 6.33]{GPV}.
\end{proof}

The proposition gives results analogous to the previous lemmas but
specializing at $z=\infty$.

\begin{lemma}
  \label{lemma:special-value-infty}
  Let $a,b,c$ be rational numbers which are not integers. 
  \begin{enumerate}
  \item Let $N$ be the least common denominator of $a,b,c$. Let
    $\varphi(z)$ be a rational function, and $z_0$ a point on $\PP^1$
    where the function has a pole of order $v$. Then if
    $vc \not \in \ZZ$, for any prime ideal $\id{p}$ of $\Q(\zeta_N)$
    not dividing $N$, the trace of $\Frob_{\id{p}}$ acting on
    $\HGM((c,1),(a,b)|\varphi)$ specialized at $z_0$ equals $1$.
    
  \item Let $N$ be the least common denominator of $a,b$. Let
    $\varphi(z)$ be a rational function having a pole at $z_0$. Then
    for any prime ideal $\id{p}$ of $\Q(\zeta_N)$ not dividing $N$,
    the trace of $\Frob_{\id{p}}$ acting on
    $\HGM((1,1),(a,b)|\varphi(z))$ specialized at $z_0$ equals $1$.
    
  \end{enumerate}
\end{lemma}

At last, we have the following result regarding specializations at $z=1$.

\begin{lemma}
  \label{lemma:spec-1}
  Let $(a,b),(c,d)$ be rational generic parameters such that
  $a+b-c-d \not \in \Z$. Then for any prime ideal $\id{p}$ of $F$ not
  dividing $N$, the
    trace of $\Frob_{\id{p}}$ acting on
    $\HGM((a,b),(c,d)|1)$ equals
  \begin{equation}
    \label{eq:formula}
    \JacMot((a-c,a-d,b-c,b-d),(a,b,-c,-d,a+b-c-d))(\id{p}).
  \end{equation}
\end{lemma}
\begin{proof}
  The proof follows the lines of \cite[Theorem A.3]{GPV}. We consider the equation
  \[
\C: y^N=x^A(1-x)^B(1-zx)^C z^D,
    \]
    where $A=(d-b)N$, $B=(b-c)N$, $C=(a-d)N$ and $D=dN$.  The
    semistable model has two components that intersect in a single
    point, with equations (setting $x=1-(z-1)x'$)
  \[
    \C_1:y^N = x^A(1-x)^{B+C}, \qquad \C_2: y^N(z-1)^{-B-C} =
    (1-(z-1)x')^Ax'^B (1+zx')^Cz^D.
  \]
  The hypothesis $a+b-c-d \not \in \ZZ$ implies that $B+C=(a+b-c-d)N$
  is not divisible by $N$, so the curve $\C_2$ has singular reduction
  (it attains good reduction over the extension $\sqrt[N]{z-1}$), so
  it does not contribute to the trace. The reduction of $\C_1$
  is the curve
    \[
y^N = x^{(d-b)N}(1-x)^{(a+b-c-d)N},
\]
whose new part contributes equals
$\JacMot((b-d,c+d-a-b),(c-a))(\id{p})$. Dividing this quantity by the
normalization and using (\ref{eq:inversion}), we obtain the equality
\begin{eqnarray*}
  H_{\id{p}}((a,b),(c,d)|1)&=&\gent(-1)^{(d-b)N}\frac{J((b-d,c+d-a-b),(c-a))(\id{p})}{J((-a,-b,c,d),(c-b,d-a))(\id{p})}=\\
  &=&J((b-c,b-d,a-c,a-d),(a,b,-c,-d,a+b-c-d)).
\end{eqnarray*}
\end{proof}

\begin{lemma}
  \label{lemma:spec-1-1}
  Let $(a,b),(c,d)$ be rational generic parameters such that $a+b$ and
  $c+d$ are integers. Then for any prime ideal $\id{p}$ of $F$ not
  dividing $N$, the trace of $\Frob_{\id{p}}$ acting on
  $\HGM((a,b),(c,d)|1)$ equals $\normid{p}^{\delta}$, where
  \[
\delta =
\begin{cases}
  0& \text{ if } a \in \ZZ \text{ or }c \in \ZZ,\\
  1 & \text{ otherwise}.
\end{cases}
    \]
\end{lemma}

\begin{proof}
  The proof is similar to that of
  Lemma~\ref{lemma:special-value-infty}. We can assume without loss of
  generality that $a=-b$ and $c=-d$. Set $A=(a-c)N$, $B=-(a+c)N$,
  $C=(a+c)N=-B$ and $D=-cN$, so we need to study the reduction at $z=1$ of Euler's curve
  \[
\C:y^N=x^A(1-x)^B(1-zx)^{-B}z^{-cN}.
\]
One component is given by its reduction modulo $z-1$, namely the curve
\[
\C_1: y^N=x^A.
\]
To get the other component, let $x'=\frac{(1-x)}{(z-1)}$, so the equation for $\C$ transforms into
\[
y^N=(1-(z-1)x')^Ax'^B(zx'-1)^{-B}z^{-cN}.
\]
Its reduction modulo $z-1$ gives the curve
\[
\C_2:y^N=x'^B(x'-1)^{-B}.
\]
Both curves $\C_1$ and $\C_2$ have genus zero, and intersect at $N$
points defined over $F$, giving just the classical Steinberg
representation. To prove the formula we need to compute the
contribution from the Jacobi motive.
\[
\JacMot((-a,a,-c,c),(a+c,-a-c))=\frac{g(\psi,\gent^{-aN})g(\psi,\gent^{aN})g(\psi,\gent^{-cN})g(\psi,\gent^{cN})}{g(\psi,\gent^{(a+c)N})g(\psi,\gent^{-(a+c)N})}.
\]
Recall from~(\ref{eq:inversion}) that
\[
  g(\psi,\chi)g(\psi,\chi^{-1}) =\chi(-1) \cdot
  \begin{cases}
    \normid{p}& \text{ if }\chi \neq 1,\\
    1 & \text{otherwise}.
  \end{cases}
  \]
  Let $\delta$ as in the statement. Since the parameters are generic,
  $a+c \not \in \ZZ$, so
\[
  \JacMot((-a,a,-c,c),(a+c,-a-c))=\normid{p}^\delta.
  \]
\end{proof}

\section{Kummer's $S_4$ transformations}
\label{section:Kummer}
In \cite[\S 8]{MR1578088}, Kummer listed twenty four transformations
$\{T_i\}_{i=1}^{24}$ that send a solution of the differential
equation~(\ref{eq:diff-eq}) to another solution. For each
transformation he wrote down the explicit relations satisfied between
the corresponding hypergeometric functions. In this section we compute
the analogous relations between the hypergeometric motives.

\begin{definition}
  Let $\alpha=\frac{r}{N}$ be a rational number, with $r,N$ coprime
  integers. Define the character
  $$\theta_\alpha:\Gal(\overline{\Q(z)}/\Q(z,\zeta_N)) \to \overline{\QQ}^\times,$$
  to be the character that factors through
  $\Gal(\QQ(\sqrt[N]{z},\zeta_N)/\QQ(z,\zeta_N))$ whose value at $\sigma$ equals
  \[
\theta_{\alpha}(\sigma) = \left(\frac{\sigma(\sqrt[N]{z})}{\sqrt[N]{z}}\right)^r.
\]
\label{def:theta}
\end{definition}
The character is unramified outside $\{0,\infty\}$. Let
$z_0 \in \Q \setminus \{0\}$ and let $\id{p}$ be a prime ideal of
$\Q(\zeta_N)$ of norm $q$ not dividing $Nz_0$. Then an easy
computation (see \S 6.5 of \cite{GPV}) shows that the specialization of
$\theta_{\alpha}$ at $z_0$ satisfies
\begin{equation}
  \label{eq:eta-weil}
\theta_{\alpha}(\Frob_{\id{p}})= \gent(z_0)^r.
\end{equation}

Similarly, for each rational number $\alpha$ there is a character ramified only
at $\{1,\infty\}$.
\begin{definition}
  Let $\alpha=\frac{r}{N} \in \Q$ as before. Let
  $\eta_\alpha : \Gal(\overline{\Q(z)}/\Q(z,\zeta_N))\to
  \overline{\Q}^\times$ be the character defined by 
  \begin{equation}
    \label{eq:chi-def}
    \eta_\alpha(\sigma)=\left(\frac{\sigma(\sqrt[N]{1-z})}{\sqrt[N]{1-z}}\right)^r.
  \end{equation}
\label{def:eta}
\end{definition}
\begin{theorem} Let $a,b,c$ be rational numbers satisfying that
  $(a,b),(c,1)$ is generic. Let $N$ denote their least common
  denominator and let $F=\Q(\zeta_N)$.  Kummer's transformations imply
  the equality of the following twenty four hypergeometric motives
  defined over $F$.
  \begin{enumerate}[i)]
  \item $\HGM((a,b),(c,1)|z)$,
    
  \item $\HGM((c-a,c-b),(c,1)|z)\otimes \eta_{c-a-b}$,
    
  \item $\HGM((a-c,b-c),(-c,1)|z)\otimes \JacMot((a-c,b-c,c),(a,-c,b))\theta_{-c}$,
    
  \item $(-1)^{c}\HGM((-a,-b),(-c,1)|z)\otimes \JacMot((a-c,-b,c),(a,-c,c-b))\theta_{-c}\eta_{c-a-b}$,
    
  \item $(-1)^a\HGM((a,b),(a+b-c,1)|1-z)\otimes \JacMot((a-c,c-a-b),(-c,c-b))$,
  \item $(-1)^{b+c}\HGM((a-c,b-c),(a+b-c,1)|1-z)\otimes \JacMot((a-c,b-c,c-a-b),(a,-c,-a))\theta_{-c}$,
    
  \item $(-1)^{a}\HGM((c-a,c-b),(c-a-b,1)|1-z)\otimes \JacMot((a-c,c-a,a+b-c),(a,-c,b))\eta_{c-a-b}$,
    
  \item $(-1)^{b+c}\HGM((-a,-b),(c-a-b,1)|1-z)\otimes \JacMot((-b,a+b-c),(a,-c))\theta_{-c}\eta_{c-a-b}$,
    
  \item $\HGM((a,a-c),(a-b,1)|\frac{1}{z})\otimes \JacMot((a-c,b-a),(b,-c))\theta_{-a}$,
    
  \item $\HGM((b,b-c),(b-a,1)|\frac{1}{z})\otimes \JacMot((b-c,a-b),(a,-c))\theta_{-b}$,
    
  \item $(-1)^{c}\HGM((-a,c-a),(b-a,1)|\frac{1}{z})\otimes \JacMot((a-c,c-a,a-b),(a,-c,c-b))\theta_{a-c}\eta_{c-a-b}$,
    
  \item $(-1)^{c}\HGM((-b,c-b),(a-b,1)|\frac{1}{z})\otimes \JacMot((b-c,c-b,b-a),(b,-c,c-a))\theta_{b-c}\eta_{c-a-b}$,
    
  \item $(-1)^a\HGM((a,c-b),(a-b,1)|\frac{1}{1-z})\otimes \JacMot((a-c,b-a),(b,-c))\eta_{-a}$,
    
  \item $(-1)^b\HGM((b,c-a),(b-a,1)|\frac{1}{1-z})\otimes \JacMot((b-c,a-b),(a,-c))\eta_{-b}$,
    
  \item $(-1)^{b+c}\HGM((a-c,-b),(a-b,1)|\frac{1}{1-z})\otimes \JacMot((a-c,-b,b-a),(a,-c,-a))\theta_{-c}\eta_{c-a}$,
    
  \item $(-1)^{a+c}\HGM((b-c,-a),(b-a,1)|\frac{1}{1-z})\otimes \JacMot((b-c,-a,a-b),(b,-c,-b))\theta_{-c}\eta_{c-b}$,
    
  \item $\HGM((a,c-b),(c,1)|\frac{z}{z-1})\otimes \eta_{-a}$,
    
  \item $\HGM((b,c-a),(c,1)|\frac{z}{z-1})\otimes \eta_{-b}$,
    
  \item $(-1)^c\HGM((a-c,-b),(-c,1)|\frac{z}{z-1})\otimes \JacMot((a-c,-b,c),(a,-c,c-b))\theta_{-c}\eta_{c-a}$,
    
  \item $(-1)^c\HGM((b-c,-a),(-c,1)|\frac{z}{z-1})\otimes \JacMot((b-c,-a,c),(b,-c,c-a))\theta_{-c}\eta_{c-b}$,
    
  \item $(-1)^a\HGM((a,a-c),(a+b-c,1)|\frac{z-1}{z})\otimes \JacMot((a-c,c-a-b),(c-b,-c))\theta_{-a}$,
    
  \item $(-1)^b\HGM((b,b-c),(a+b-c,1)|\frac{z-1}{z})\otimes \JacMot((b-c,c-a-b),(c-a,-c))\theta_{-b}$,
    
  \item $(-1)^{a}\HGM((-a,c-a),(c-a-b,1)|\frac{z-1}{z})\otimes \JacMot((a-c,c-a,a+b-c),(a,-c,b))\theta_{a-c}\eta_{c-a-b}$,
    
  \item $(-1)^{b}\HGM((-b,c-b),(c-a-b,1)|\frac{z-1}{z})\otimes \JacMot((b-c,c-b,a+b-c),(b,-c,a))\theta_{b-c}\eta_{c-a-b}$,
  \end{enumerate}
  where $(-1)^a$ is the quadratic character of $F$ whose value at a prime ideal $\id{p}$ equals $\gent(-1)^{Na}$.
\end{theorem}

\begin{proof}
  For $\varepsilon,\omega,\chi$ characters of $\FF_q$ and $z \in\FF_q$, define the function
  \begin{equation}
    \label{eq:H-defi}
    H(z)=\sum_{x \in \FF_q} \varepsilon(x) \omega(1-x)\chi^{-1}(1-zx).
  \end{equation}
  Setting $\varepsilon = \gent^{aN}$, $\omega = \gent^{(c-a)N}$ and
  $\chi = \gent^{bN}$, it follows from \cite[Theorem 4.10]{GPV} that
  \begin{equation}
    \label{eq:factorization}
    H(z) =
    \frac{\varepsilon(-1)g(\psi,\varepsilon)g(\psi,\varepsilon^{-1}\omega^{-1})}{g(\psi,\omega^{-1})}
    H_{\id{p}}((a,b),(c,1)|z).
  \end{equation}
  The function $H(z)$ is a finite version of the integral defining the
  hypergeometric series. The stated transformation results for $H(z)$
  follow from applying each change of variables suggested in
  \cite{MR752488} (Table 1). The order of the transformations listed
  in Dwork's table is not the same as Kummer's original article, so to
  avoid confusions, in Table~\ref{table:characters} we reordered the
  entries of Dwork's table. The change of variables corresponds to set
  $x=f(y)$ for each entry. Although this transformation might not be
  well defined at one of the points $\{0,1,1/z\}$ the term in the
  definition of $H(z)$ takes the value $0$ at all such points.

\begin{table}
  \begin{tabular}{|c|c||c|c||c|c||c|c|}
    \hline
    $1$ & $y$ & $7$ & $\frac{1}{1+y(z-1)}$ & $13$ & $\frac{y}{y+z-1}$ & $19$ & $\frac{1}{z(1-y)}$\\
\hline
    $2$ & $\frac{1-y}{1-zy}$ & $8$ & $\frac{1+y(z-1)}{z}$ & $14$ & $\frac{1}{1-y}$ & $20$ & $1+\frac{1-z}{zy}$\\
\hline
    $3$ & $\frac{1}{zy}$ & $9$ & $\frac{y}{z}$ & $15$ & $\frac{1-y}{z}$ & $21$ & $\frac{y}{z(y-1)}$\\
\hline
    $4$ & $\frac{y-z^{-1}}{y-1}$ & $10$ & $\frac{1}{y}$ & $16$ & $\frac{y+z-1}{yz}$ & $22$ & $\frac{y-1}{y}$\\
\hline
    $5$ & $\frac{y}{y-1}$  & $11$ & $\frac{z-y}{z(1-y)}$ & $17$ & $\frac{y}{1+zy-z}$ & $23$ & $1+\frac{(1-z)y}{z}$\\
\hline
    $6$ & $\frac{y-1}{zy}$ & $12$ & $\frac{y-1}{y-z}$ & $18$ & $1-y$ & $24$ & $\frac{1}{z+(1-z)y}$\\
\hline
  \end{tabular}
	\caption{Change of variables \label{table:characters}}
\end{table}

For example, the second formula corresponds to the change of variables
$x = \frac{1-y}{1-zy}$. Then \eqref{eq:H-defi} becomes
  \[
    \sum_{y \in \FF_q} \varepsilon\left(\frac{1-y}{1-zy}\right)\omega\left(\frac{y(1-z)}{1-zy}\right) \chi^{-1}\left(\frac{1-z}{1-zy}\right) = \omega\chi^{-1}(1-z) \sum_{y \in \FF_q} \omega(y)\varepsilon(1-y)(\varepsilon \omega \chi^{-1})^{-1}(1-zy),
  \]
  which matches the series for the parameters $(c-a,c-b),(c,1)$ at $z$
  twisted by $\eta_{c-a-b}$ (corresponding to the value outside the
  summation). The first term in \eqref{eq:factorization} is the same
  for both parameters, proving the formula. Similar computations
  prove all other cases.
\end{proof}

\begin{remark}
  As a safety check, there is code available at the author's web page
  \url{https://sweet.ua.pt/apacetti} to verify each of these twenty
  four transformations for specializations of the variable.
\end{remark}

\begin{remark}
  As is well known, the twenty four transformations found by Kummer
  are not closed under composition, i,e. they do not form a group (see
  \S2 \cite{MR752488} for a description of them). It is not hard to
  compute the group generated by these 24 transformations. If we
  forget the ``twist'' appearing in the previous formulas, we can
  represent each transformation by the way it acts on the parameters
  $(a,b,c)$ (corresponding to a linear transformation in $\GL_3(\ZZ)$)
  and its action on $z$ (a M\"obius transformation of
  $\PP^1\setminus\{0,1,\infty\}$). The resulting group $H$ has order
  $144$, corresponding to the small group number $189$, isomorphic to
  $C_2 \times (C_3 \rtimes S_3$). Then the previous list of 24
  transformations can be enlarged to a list containing $144$ ones.
\end{remark}

\section{Hypergeometric relations and covers}

As mentioned in the introduction, Goursat in \cite{MR1508709} proved
many quadratic transformation formulas satisfied by Gauss'
hypergeometric function coming from hypergeometric pairs
$(\varphi,\rho)$ (as well as some larger degree ones). The
strategy he used is not that different from the two steps mentioned
before for hypergeometric motives (our strategy is largely influenced by his):
\begin{enumerate}
\item Prove that both sides of the equality satisfy the same differential equation.
  
\item Prove that the value of both functions at a well chosen point
  coincide.
\end{enumerate}

Corollary~\ref{coro:unicity} together with the results obtained in \S
\ref{sec:Special} can probably be used to prove arithmetic versions of
each of the seventy six transformations discovered by Kummer in
\cite{MR1578088} as well as the one hundred and thirty seven ones
discovered by Goursat in \cite{MR1508709}. We content ourselves
studying some of them.

\begin{theorem}
\label{theorem:isom}
Let $(\varphi,\rho)$ be a hypergeometric pair, where
$\rho = \HGM((a,b),(c,1)|z)$. Let $a',b',c'$ be rational numbers such
that the monodromy representation of $\hgm$ is isomorphic to the
monodromy representation of $\HGM((a',b'),(c',1)|\varphi(z))$. Suppose
that $\varphi$ vanishes at $z=0$ with order $v>0$ and one of the
following two hypothesis is satisfied:
\begin{enumerate}
\item $c \not \in \ZZ$ and $vc' \not \in\ZZ$,
  
\item $c, c' \in \ZZ$.
\end{enumerate}
Then there is an isomorphism of hypergeometric motives
  \begin{equation}
    \label{eq:hyper-isom}
    \HGM((a,b),(c,1)|z) \simeq \HGM((a',b'),(c',1)|\varphi(z)).
  \end{equation}
\end{theorem}

\begin{proof} Since both motives have the same monodromy
  representation, Corollary~\ref{coro:unicity} implies that they are a
  twist of each other. Then (as explained in \S \ref{sec:HGM}) it is
  enough to prove that they are isomorphic for a particular
  specialization of the parameter $z$ (like $z=0$). Under the first
  hypothesis, Lemma~\ref{lemma:spec-0} implies that the trace of the
  Frobenius automorphism $\Frob_{\id{p}}$ acting of the left hand side motive
  equals $1$ for any prime ideal $\id{p}$ of $F$ not dividing $N$.

  When the first hypothesis is satisfied, Lemma~\ref{lemma:spec-0}
  also implies that the right hand side also evaluates to $1$ at
  $z=0$. When the second hypothesis holds the result follows from
  Lemma~\ref{lemma:spec0-1}.
\end{proof}

As an application to Example~\ref{ex:1} we have the following.

\begin{corollary}
  Let $a,b$ be rational numbers satisfying:
  \begin{itemize}
  \item They are not integers,
    
  \item $\frac{a-b+1}{2}, \frac{a+1}{2}$ and $\frac{b+1}{2}$ are not integers.
  \end{itemize}
  Then the arithmetic version of (\ref{eq:1}) is the following
  isomorphism between hypergeometric motives
\begin{equation}
  \label{eq:hypergeom-isom}
  \HGM\left((a,b),\left(\frac{a+b+1}{2},1\right)|z\right) = \HGM\left(\left(\frac{a}{2},\frac{b}{2}\right),\left(\frac{a+b+1}{2},1\right)|4z(1-z)\right).
\end{equation}
\end{corollary}

\begin{proof}
  The hypothesis imply that both parameters are generic. We already
  proved in the introduction that the monodromies of both motives are
  isomorphic. The result then follows from the last theorem (as the
  trace at a Frobenius element $\Frob_{\id{p}}$, for $\id{p}$ a prime
  ideal not dividing $N$, of both sides at $z=0$ equals $1$).
\end{proof}

\begin{remark}
  The isomorphism~(\ref{eq:hypergeom-isom}) has deep implications.
  Let $N$ be the least common multiple of the denominators of
  $\frac{a}{2}$, $\frac{b}{2}$ and $\frac{a+b+1}{2}$, and let $N'$ be
  the least common multiple of the denominators of $a$,
  $b$ and $\frac{a+b+1}{2}$. Clearly $N'\mid N$, but they
  might be different.  The field of definition of the motive
  $M_1:=\HGM\left(\left(\frac{a}{2},\frac{b}{2}\right),\left(\frac{a+b+1}{2},1\right)|z\right)$
  is generically $K(\zeta_N)$, while the field of definition of
  $M_2:=\HGM\left((a,b),\left(\frac{a+b+1}{2},1\right)|z\right)$ is
  $\Q(\zeta_{N'})$. The isomorphism then implies that if $N \neq N'$,
  the motive $M_1$ can be defined over an extension smaller than the
  expected one. However, this is only true for specializations lying
  in the image of $\varphi$ (a thin set). This goes in accordance with
  the expectation that for most specializations of the parameter
  (outside a thin set), the coefficient field of the motive is the one
  conjectured by David Roberts and Fernando Rodriguez Villegas (see
  \cite[Conjecture 3.2]{GPV}). See also Example 7.21 of \cite{GPV} for
  a more naive example of this phenomena.
\end{remark}

The arithmetic version of Example~\ref{ex:2} is the following.
\begin{proposition}
  Let $a, b$ be rational numbers satisfying:
  \begin{itemize}
  \item They are not integers,
    
  \item $\frac{a-b+1}{2}$ is not an integer.
  \end{itemize}
  Then the arithmetic version of (\ref{eq:67}) are the isomorphisms
\begin{equation}
 \HGM\left((a,b),\left(\frac{a+b+1}{2},1\right)|z\right) = \HGM\left((a,b),\left(\frac{a+b+1}{2},1\right)|1-z\right),
\end{equation}
and
\begin{multline}
  \HGM\left((a,b),\left(\frac{a+b+1}{2},1\right)|z\right) = \eta_{(1-a-b)/2} \Jac\left(\left(\frac{a+b+1}{2},\frac{a+b+1}{2}\right),(a,b)\right) \\ \HGM\left(\left(\frac{a-b+1}{2},\frac{b-a+1}{2}\right),\left(\frac{1-a-b}{2},1\right)|1-z\right).
\end{multline}
\end{proposition}

\begin{proof}
  In Example~\ref{ex:2} we proved that the monodromy representation of
  the three motives involved are isomorphic. The result then follows
  by evaluating at $z=0$ and using Lemma~\ref{lemma:spec-0} (or
  Lemma~\ref{lemma:spec0-1}) and Lemma~\ref{lemma:spec-1}.
\end{proof}

Let us state two more examples of the method.

\begin{proposition}
  \label{thm:hgm-equality-1}
  Let $(a,b)$ be rational numbers which are not integers.
  Then the following hypergeometric motives are isomorphic:
  \begin{multline}
    \HGM((a,b),(\frac{a+b+1}{2},1)|z) =\\ \JacMot\left(\left(\frac{a+1}{2},\frac{b+1}{2}\right),\left(\frac{a+b+1}{2},\frac{1}{2}\right)\right) \HGM\left(\left(\frac{a}{2},\frac{b}{2}\right),\left(\frac{1}{2},1\right)\left|(1-2z)^2\right.\right).    
  \end{multline}
\end{proposition}

\begin{proof}
  Consider the map $\varphi(z)=(2z-1)^2$ ramified at $\{0,\infty\}$
  (the preimage of $0$ being $\frac{1}{2}$). The preimage of $1$ being
  the two points $0$ and $1$. It then follows from a simple
  verification (using the monodromy matrices definition given in
  (\ref{eq:monodromy}) and the fact that our parameters are generic)
  that the two monodromy representations are indeed
  isomorphic. Specialize both motives at $z=0$, so the left hand side
  is evaluated at $0$ while the right one at
  $1$. Lemma~\ref{lemma:spec-0} (if $\frac{a+b+1}{2}$ is not an
  integer) or Lemma~\ref{lemma:spec0-1} (in the other case) imply that
  \[
    \HGM\left((a,b),\left(\frac{a+b+1}{2},1\right)\left|0\right.\right)=1.
  \]
  Lemma~\ref{lemma:spec-1} implies that
  \begin{multline}
    \HGM\left(\left(\frac{a}{2},\frac{b}{2}\right),\left(\frac{1}{2},1\right)|1\right)=\Jac\left(\left(\frac{a+1}{2},\frac{b+1}{2}\right),\left(\frac{a+b+1}{2},\frac{1}{2}\right)\right) = \\ \Jac\left(\left(\frac{a+b+1}{2},\frac{1}{2}\right),\left(\frac{a+1}{2},\frac{b+1}{2}\right)\right)^{-1},    
  \end{multline}
  so the Jacobi motive relating them equals
  $\Jac\left(\left(\frac{a+1}{2},\frac{b+1}{2}\right),\left(\frac{a+b+1}{2},\frac{1}{2}\right)\right)$.
\end{proof}

\begin{remark}
  Note that the cover $\varphi(z)=(2z-1)^2$ sends $1$ to $1$, so we
  can apply the same reasoning as before for the value $z=1$, in which
  case the same result follows from the easy equality
  \begin{equation}
    \label{eq:equality}
    \Jac\left(\left(\frac{1+a+b}{2},\frac{1-a-b}{2}\right),\left(\frac{1-a+b}{2},\frac{1+a-b}{2}\right)\right)=1.
  \end{equation}
\end{remark}

The following is the arithmetic version of formula 69 in \cite{MR1578088}.

\begin{theorem}
  \label{thm:kummer-69}
  Let $(a,b)$ be rational integers satisfying:
  \begin{itemize}
  \item $a,b$ are not integers,
    
  \item The numbers $\frac{-a+b+1}{2}$, $\frac{a-b+1}{2}$ and $\frac{b+1}{2}$ are not integers.
  \end{itemize}
  Then the following hypergeometric motives are isomorphic:
  \begin{equation*}
    \HGM\left((a,b),\left(\frac{a+b+1}{2},1\right)|z\right) = \HGM\left(\left(\frac{a}{2},\frac{a+1}{2}\right),\left(\frac{a+b+1}{2},1\right)\left|\frac{4z^2-4z}{4z^2-4z+1}\right.\right)\chi_{-a},    
  \end{equation*}
where $\chi_{a/N}$ is the Hecke character corresponding to the extension $\Q(\zeta_N,\sqrt[N]{(2-z)^a})/\Q(\zeta_N,z)$.
\end{theorem}

\begin{proof}
  The map $\varphi(z)=\frac{4z^2-4z}{4z^2-4z+1}$ is ramified at $1$
  and $\infty$. It maps $\{0,1\} \to 0$, $\infty \to 1$ and
  $1/2 \to \infty$. The monodromy matrices of the right hand side
  motive are:
  \[
M_0 = M_1=
\begin{pmatrix}
  e^{\frac{-a-b-1}{2}} & 0\\ 0 & 1
\end{pmatrix},
\qquad M_{1/2} =
\begin{pmatrix}
  e^a & 0 \\ 0 & e^a
\end{pmatrix},
\qquad
M_\infty =
\begin{pmatrix}
  e^{-a+b} & 0 \\ 0 & 1
\end{pmatrix}.
\]
Twisting by $\eta_{-a}$ kills the ramification at $1/2$ and transforms
the ramification at $\infty$ into
$\begin{pmatrix} e^a & 0\\ 0 & e^b\end{pmatrix}$, which
matches the monodromy representation of the left hand side. The result
then follows by evaluating both quantities at $z=0$ (using
Theorem~\ref{theorem:isom}).
\end{proof}

At last, let us study an example of a degree $4$ cover (formula 131 of \cite{MR1508709}).

\begin{theorem}
  Let $a$ be a rational integer, such that none of $4a$, $2a+1/6$,
  $a+1/2$ and $a-1/6$ is an integer. Then the following motives are isomorphic
  \label{thm:deg4}
  \[
\HGM\left(\left(4a,2a+1/6\right),(2/3,1)|z\right) =\eta_{-3a}\HGM\left(\left(a,1/6-a\right),(2/3,1)\left|-\frac{(z+8)^3z}{64(1-z)^3}\right.\right)
    \]
  \end{theorem}

  \begin{proof}
    The cover $\varphi=-\frac{(z+8)^3z}{64(1-z)^3}$ ramifies at the
    points $0, 1$ and $\infty$. The monodromy matrices of the left hand side motive are
    \[
N_0=
\begin{pmatrix}
  \exp(1/3) & 0\\
  0 & 1
\end{pmatrix},\qquad
N_1=
\begin{pmatrix}
  1 &0\\
  0 & -1
\end{pmatrix},
\qquad
N_\infty =
\begin{pmatrix}
  \exp(a) & 0\\
  0 & \exp(1/6-a)
\end{pmatrix}.
      \]
  The two points above $1$ are $10 \pm 6\sqrt{2}$ and have
  ramification degree $2$, so the monodromy matrix of the pullback is
  trivial for them both. The preimage of $0$ consists of the points
  $0$ (unramified) and $-8$ with ramification degree $3$, so $-8$ is
  unramified for the pullback. At last, the preimage of $\infty$
  consists of the points $1$ (with ramification degree $3$) and
  $\infty$, which is unramified. Then the monodromy matrices at $0,1$
  and $\infty$ of the pullback are (respectively)
  \[
    \begin{pmatrix}
  \exp(1/3) & 0\\
  0 & 1
\end{pmatrix},\qquad
\begin{pmatrix}
  \exp(3a) &0\\
  0 & \exp(-3a)
\end{pmatrix},
\qquad
\begin{pmatrix}
  \exp(a) & 0\\
  0 & \exp(1/6-a)
\end{pmatrix}.
 \]
 The twist by $\eta_{-3a}$ multiplies the monodromy matrix at $1$ by
 $\exp(-3a)$ and the one at $\infty$ by $\exp(3a)$ so both monodromy
 representations are isomorphic. The result follows from
 Theorem~\ref{theorem:isom}.
\end{proof}

\begin{remark}
  Each stated isomorphism gives a non-trivial relation between sums
  and products of Gauss sums (corresponding to the trace of Frobenius
  of the right and the left hand side of the isomorphism). Is there a
  direct way to prove such type of formulas?
\end{remark}

\section{Applications to Diophantine equations}

In \cite{Darmon} Darmon presented a general program to study solutions
of the generalized Fermat equation
\begin{equation}
  \label{eq:Fermat}
  Ax^p + By^q = Cz^r.
\end{equation}
It is expected that for fixed values of $A,B,C$, there are only
finitely many \emph{primitive} solutions for any triple of exponents
$(p,q,r)$ satisfying $\frac{1}{p}+\frac{1}{q}+\frac{1}{r}<1$. Recall
the following definition.
\begin{definition}
  A solution $(\alpha,\beta,\gamma)$ of (\ref{eq:Fermat}) is called
  primitive if $\gcd(\alpha,\beta,\gamma)=1$.
\end{definition}
Darmon's program consists on attaching to the equation exponents
(namely $p,q,r$) what he calls a \emph{Frey representation}, 
a representation of $\Gal(\overline{\Q(z)}/F(z))$ (for some number
field $F$) into $\GL_2(\FF)$ for some finite field
$\FF$. If $(\alpha,\beta,\gamma)$ is a solution to \eqref{eq:Fermat},
then the specialization of the family at
$z_0:=\frac{A\alpha^p}{C\gamma^q}$ corresponds to a finite extension
of $F$ with little ramification.

Following the ideas used in Wiles' proof of Fermat's last theorem, the
representation should be the reduction of a representation attached to
a modular form, and (assuming various conjectures) such modular form
should not exist if one of the exponents (say $p$) is sufficiently
large. Proving the conjectured missing results is nowadays a deep problem.

In \cite{GP} the authors gave a similar approach replacing the finite
field $\FF$ with a $p$-adic field, using the theory of hypergeometric
motives. Hypergeometric motives give more flexibility to prove some of
the expected properties, but the theory is not completely understood.
For example, if $(a,b),(c,d)$ are generic rational parameters, and $N$
is their least common denominator, we do not understand the action of
inertia on $\hgm$ at primes of $F=\Q(\zeta_N)$ dividing $N$.

While studying the family of exponents $(p,p,q)$, the motive is part
of the middle cohomology of an hyperelliptic curve, where the
reduction type at odd primes is well understood (see for example
\cite{Tim} and \cite{2410.21134}). For the family $(q,q,p)$ this is
not the case, the motive appears in Euler's curve which is
superelliptic, a case where the description of wild inertia is still not fully
understood.

However, in the remarkable article \cite{2205.15861}, the authors manage to
relate a solution of the equation
\begin{equation}
  \label{eq:qqp}
  x^q + y^q = z^p,
\end{equation}
with an hyperelliptic curve $\C_t$. The constructed curve gives a
\emph{Frey representation} (as defined by Darmon in \cite{Darmon}) but
ramifies at $\{\pm 2i, \infty\}$ (see \cite[Proposition
2.32]{2205.15861}). An extra transformation is needed to transform
this set into $\{0,1,\infty\}$. Still, if $(\alpha,\beta,\gamma)$ is a solution of
\eqref{eq:qqp}, the curve $\C_t$ is not evaluated at the usual point
$t_0=\frac{\alpha^q}{\gamma^p}$, but at a different one satisfying a
``mysterious'' algebraic relation with $t_0$. The purpose of this
section is to explain how to attach naturally an hyperelliptic curve
to Fermat equation with exponents $(q,q,p)$ from two different
hypergeometric relations. This approach might prove useful while
studying other families of exponents.

\vspace{2pt}

The three hypergeometric motives involved are:
\begin{enumerate}
\item The hypergeometric motive
  $\HGM\left(\left(\frac{1}{2q},-\frac{1}{2q}\right),(1,1)|z\right)$ (that we
  denote by $\Mp$ to ease notation) associated to the exponents
  $(p,p,q)$ in \cite{GP} (see also \cite{Darmon}).
  
\item The hypergeometric motive
  $\HGM\left(\left(\frac{1}{2q},-\frac{1}{2q}\right),\left(\frac{s}{q},-\frac{s}{q}\right)|z\right)$
  (that we denote by $\Mq$) attached to the exponents
  $(q,q,p)$ in \cite{GP}, for any $s \in \{1,\ldots, q-1\}$.
  
\item The hypergeometric motive   $\M =
  \HGM\left(\left(\frac{1}{2}-\frac{1}{4q},\frac{1}{4q}\right)(1,1)|z\right)$.
\end{enumerate}
For $N$ a positive integer, denote by $\zeta_{N}=\exp(\frac{1}{N})$, a
primitive $N$-th root of unity. The monodromy matrices around
$\{0,1,\infty\}$ of the three motives are given in
Table~\ref{table:monodromies}.
\begin{table}
\begin{tabular}{|c|c|c|c|}
\hline
Motive & $M_0$ & $M_1$ & $M_\infty$\\
\hline\hline
  $\Mp$ &
 $\begin{pmatrix}
  1 & 1\\
  0 & 1
 \end{pmatrix}$ &
 $\begin{pmatrix}
  1 & 1\\
  0 & 1
 \end{pmatrix}$ &
  $ \begin{pmatrix}
  \zeta_{2q} & 0\\
  0 & \zeta_{2q}^{-1}
 \end{pmatrix}$\\
\hline
  $\Mq$ &
 $\begin{pmatrix}
  \zeta_q^s & 0\\
  0 & \zeta_q^{-s}
 \end{pmatrix}$ &
 $\begin{pmatrix}
  1 & 1\\
  0 & 1
 \end{pmatrix}$ &
  $ \begin{pmatrix}
  \zeta_{2q} & 0\\
  0 & \zeta_{2q}^{-1}
 \end{pmatrix}$\\
\hline
  $\M$ &
 $\begin{pmatrix}
  1 & 1\\
  0 & 1
 \end{pmatrix}$ &
 $\begin{pmatrix}
  -1 & 0\\
  0 & 1
 \end{pmatrix}$ &
  $ \begin{pmatrix}
  -\zeta_{4q}^{-1} & 0\\
  0 & \zeta_{4q}
 \end{pmatrix}$\\
\hline
\end{tabular}
\caption{Monodromy matrices\label{table:monodromies}}
\end{table}
Ideally, we seek for degree $2$ covers $\pi_1$ and $\pi_2$ 
such that the relation between the three motives is explained by 
the following diagram
\begin{equation}
  \label{eq:covers}
  \xymatrix{
    \Mp \ar@{->}[dr]^{\pi_1} & & \Mq \ar@{->}[dl]_{\pi_2}\\
    & \M
}  
\end{equation}
By looking at Table~\ref{table:monodromies} it is clear that some
adjustments are needed, since a degree $2$ cover tends to have the
effect of duplicating a monodromy in the pullback, but the monodromy
matrices of $\Mq$ are all different. 

Let $\theta_{-1/4}$ be the order $4$ character of $\Gal_{\Q(i,z)}$ of
Definition~\ref{def:theta} (corresponding to the extension
$\Q(i,\sqrt[4]{z})/\Q(i,z)$), whose monodromy matrix at $0$ equals $\exp(-1/4)$ and at $\infty$ equals $\exp(1/4)$.  Set $s_0 = \frac{q+1}{2}$. Then the
monodromy matrices at $0$, $1$ and $\infty$ of the twist
$\Mqs\otimes \theta_{-1/4}$ are respectively
\[
\begin{pmatrix}
  \zeta_{4q}^{q+2} & 0\\
  0 & \zeta_{4q}^{q-2}
 \end{pmatrix}, \qquad
 \begin{pmatrix}
  1 & 1\\
  0 & 1
 \end{pmatrix},  \qquad
   \begin{pmatrix}
  \zeta_{4q}^{q+2} & 0\\
  0 & \zeta_{4q}^{q-2}
 \end{pmatrix}.
\]

\subsection{The map $\pi_1$:} Let $\pi_1$ be the degree $2$ cover of
$\PP^1$ given by the map $\pi_1(z)=4z(1-z)$. For $b$ a rational
number, let $\eta_b$ be the character of Definition~\ref{def:eta}.

\begin{proposition}
  Let $a,b$ be a rational numbers such that $a$ is not an integer and
  $b \pm a \not \in \ZZ$. Then
  \[
\eta_{-b}\HGM((a,-a),(b,1)|z) \simeq \HGM\left(\left(\frac{b-a}{2},\frac{a+b+1}{2}\right)(b,1)|4z(1-z)\right).
    \]
\label{prop:second-formula}
  \end{proposition}

\begin{proof}
  The proof mimics the previous section ones. The monodromy
  matrices for the hypergeometric motive on the left hand side (if
  $b$ is not an integer) are
  \begin{equation}
    M_0=\begin{pmatrix}
          \exp(-b) & 0\\
          0 & 1
        \end{pmatrix}, \qquad
        M_1=
        \begin{pmatrix}
          \exp(-b) & 0\\
          0 & 1
        \end{pmatrix}, \qquad
        M_\infty = \begin{pmatrix}
          \exp(b+a) & 0\\
          0 & \exp(b-a)
        \end{pmatrix}.
  \end{equation}
The monodromy matrices for the right hand side hypergeometric motive are
\begin{equation}
  N_0=
  \begin{pmatrix}
    \exp(-b) & 0\\
    0 & 1
  \end{pmatrix}, \qquad
  N_1=\begin{pmatrix}
        -1 & 0\\
        0 & 1
  \end{pmatrix}
  , \qquad
  N_\infty =
  \begin{pmatrix}
    \exp(\frac{b-a}{2}) & 0\\
    0 & \exp(\frac{a+b+1}{2})
  \end{pmatrix}.
\end{equation}
The cover $\pi_1$ is ramified at the points $\{1,\infty\}$ (with preimages
$1/2$ and $\infty$ respectively), and the preimage of $0$ are the two
points $\{0,1\}$. Then the monodromy representations are isomorphic. The
result then follows from Theorem~\ref{theorem:isom}. The case $b=1$
follows similarly.
\end{proof}
The left part
of~(\ref{eq:covers}) corresponds to $a=\frac{1}{2q}$ and $b=1$ in the
last proposition, getting the equality
\begin{equation}
  \label{eq:1-1}
  \HGM\left(\left(\frac{1}{2q},\frac{-1}{2q}\right),(1,1)|z\right) = \HGM\left(\left(\frac{1}{2}-\frac{1}{4q},\frac{1}{4q}\right),(1,1)|4z(1-z)\right)
\end{equation}

\subsection{The map $\pi_2$:} Let $(a,b),(c,d)$ be generic parameters
and let $N$ be their least common multiple. Since the motive is
defined over $F=\Q(\zeta_N)$, the Galois group $\Gal(\Q(\zeta_N)/\Q)$
acts on it. If $\sigma \in \Gal(\Q(\zeta_N)/\Q)$ satisfies that
$\sigma(\zeta_N)=\zeta_N^j$, then the motive $\sigma(\hgm)$ equals the
motive $\HGM((ja,jb),(jc,jd)|z)$ (as proved in \cite[Proposition
4.6]{GPV}). Instead of looking at the
motive $\M$, consider its Galois conjugate $\widetilde{\M}$ by the
element $\sigma$ corresponding to $j = q+2 \in (\ZZ/2q)^\times$.

Let $\pi_2$ be the degree $2$ cover of $\PP^1$ given by $\pi_2(z)=\frac{-(z-1)^2}{4z}$.

\begin{proposition}
  For $q$ an odd prime, and $\sigma$ as before, the following motives are isomorphic
\label{prop:isom-diof-2}
\[
  \HGM\left(\left(\frac{1}{2q},-\frac{1}{2q}\right),\left(\frac{q+1}{2q},-\frac{q+1}{2q}\right)|z\right) \otimes \theta_{-\frac{1}{4}}\simeq \HGM\left(\left(\frac{1}{2}-\frac{1}{4q},\frac{1}{4q}\right),(1,1)\left|\frac{-(z-1)^2}{4z}\right.\right)^\sigma (-1),
\]
  where the $(-1)$ denotes a Tate twist by the inverse of the cyclotomic character.
\end{proposition}

\begin{proof}
  Since the character $\theta_{-\frac{1}{4}}$ has monodromy
    $\exp(-1/4)$ at $0$ and $\exp(1/4)$ at $\infty$, the
    left hand side has monodromy matrices
    \[
M_0=
\begin{pmatrix}
  \exp(\frac{q+2}{4q}) & 0\\
  0 & \exp(\frac{q-2}{4q})
\end{pmatrix}
, \quad
M_1=
\begin{pmatrix}
  1 & 1\\
  0 & 1
\end{pmatrix}
,\quad
M_\infty=
\begin{pmatrix}
  \exp(\frac{q+2}{4q}) & 0\\
  0 & \exp(\frac{q-2}{4q})
\end{pmatrix}.
\]
The motive
$\HGM\left(\left(\frac{1}{2}-\frac{1}{4q},\frac{1}{4q}\right),(1,1)\left|z\right.\right)^\sigma$
    has monodromy matrices
\[
N_0=
\begin{pmatrix}
  1 & 1\\
  0 & 1
\end{pmatrix}
, \quad
N_1=
\begin{pmatrix}
  1 & 0\\
  0 & -1
\end{pmatrix}
,\quad
N_\infty=
\begin{pmatrix}
  \exp(\frac{q+2}{4q}) & 0\\
  0 & \exp(\frac{q-2}{4q})
\end{pmatrix}.
\]
The map $\pi_2$ is ramified at $0$ and $1$, with preimages $1$ and
$-1$ respectively. The preimage of $\infty$ are the points $0$ and
$\infty$, so the two monodromy representations are isomorphic. To
prove the statement we specialize at the value $z=1$. The
specialization of the character $\theta_{-\frac{1}{4}}$ is trivial, so
the trace of the left hand side at the Frobenius element of a prime
ideal $\id{p}$ equals $\normid{p}$ by
Lemma~\ref{lemma:spec-1-1}. The value of the right hand side equals
$1$ by Lemma~\ref{lemma:spec0-1}.
\end{proof}

Set $z_0=\frac{\alpha^q}{\gamma^P}$. Then up to a twist by
$\theta_{-\frac{1}{4}}$ Proposition~\ref{prop:isom-diof-2} implies
that instead of considering the superelliptic curve studied by Darmon,
we can consider the motive $\M(1)$ evaluated at the point
$\frac{-(z_0-1)^2}{4z_0}=-\frac{\beta^{2q}}{4\alpha^q\gamma^p}$. Then
Proposition~\ref{prop:second-formula} implies that we can instead
consider the hyperelliptic curve coming from the exponents $(p,p,q)$
evaluated at $u_0$ satisfying
\[
4u_0(u_0-1)=\frac{\beta^{2q}}{4\alpha^q\gamma^p}.
  \]
  The tame primes of the hyperelliptic curve are the ones dividing
  $2\alpha\beta\gamma$ as expected, but we can get information at the
  ramification at the prime $q$ (as exploited in \cite{2205.15861} and
  \cite{2410.21134}).

\bibliographystyle{plain}
\bibliography{biblio}

\end{document}